\documentclass[12pt]{amsart}

\usepackage{amssymb,color}
\usepackage{amsfonts}
\usepackage{amsmath}
\usepackage{enumerate}

\begingroup

\newtheorem{thm}{Theorem\hskip 5mm}[section]

\newtheorem{theorem}{Theorem\hskip 5mm}[section]

\newtheorem{prop}[thm]{Proposition\hskip 5mm}

\newtheorem{cor}[thm]{Corollary\hskip 5mm}

\newtheorem{lem}[thm]{Lemma\hskip 5mm}

\newtheorem{lemma}[thm]{Lemma\hskip 5mm}

\newtheorem{exa}[thm]{Example\hskip 5mm}

\newtheorem{note}[thm]{Note\hskip 5mm}

\endgroup

\def\r{{\mathfrak r}}

\def\Sp{{\mathrm {Sp}}}

\def\al{{\alpha}}
\def\be{{\beta}}
\def\si{{\sigma}}

\def\GL{{\mathrm{ GL}}}

\def\m{{\mathfrak m}}
\def\i{{\mathfrak i}}
\def\j{{\mathfrak j}}
\def\r{{\mathfrak r}}

\def\b{{\mathfrak b}}

\begin{document}

\title[Unitary groups and ramified extensions]{Unitary groups and ramified extensions}

\author{J. Cruickshank}
\address{School of Mathematics, Statistics and Applied Mathematics , National University of Ireland, Galway, Ireland}
\email{james.cruickshank@nuigalway.ie}

\author{F. Szechtman}
\address{Department of Mathematics and Statistics, Univeristy of Regina, Canada}
\email{fernando.szechtman@gmail.com}
\thanks{The second author was supported in part by an NSERC discovery grant}



\begin{abstract} We classify all non-degenerate skew-hermitian forms defined over certain local rings, not necessarily commutative,
and study some of the fundamental properties of the associated unitary groups, including their orders when the ring in question is finite.
\end{abstract}

\maketitle

\section{Introduction}

More than half a century ago \cite{D} offered a systematic study
of unitary groups, as well as other classical groups, over fields
and division rings. Thirty five years later, \cite{HO} expanded
this study to fairly general classes of rings. In particular, the
normal structure, congruence subgroup, and generation problems for
unitary (as well as other classical) groups are addressed in
\cite{HO} in great generality. In contrast, the problem of
determining the order of a unitary group appears in \cite{HO} only
in the classical case of finite fields, as found in \cite[\S
6.2]{HO}.

General formulae for the orders of unitary groups defined over a
finite ring where 2 is unit were given later in \cite{FH}. The proofs in \cite{FH} are
fairly incomplete and, in fact, the formulae in \cite[Theorem 3]{FH} are incorrect when the involutions induced on the given
residue fields are not the identity (even the order of the classical unitary groups defined over finite fields is wrong in this case).
The argument in \cite[Theorem 3]{FH} is primarily based on a reduction homomorphism, stated without
proof to be surjective.

Recently, the correct orders of unitary groups defined over a finite local
ring where 2 is invertible were given in \cite{CHQS}, including
complete proofs. It should be noted that the forms underlying
these groups were taken to be hermitian, which ensured the
existence of an orthogonal basis. Moreover, the unitary groups
from \cite{CHQS} were all extensions of orthogonal or unitary
groups defined over finite fields.

In the present paper we study the unitary group $U_n(A)$
associated to a non-degenerate skew-hermitian form $h:V\times V\to
A$ defined on a free right $A$-module $V$ of finite rank $n$,
where $A$ is a local ring, not necessarily commutative, endowed
with an involution $*$ that satisfies $a-a^*\in\r$ for all $a\in
A$, and $\r$ is the Jacobson radical of $A$. It is also assumed
that $\r$ is nilpotent and $2\in U(A)$, the unit group of $A$.
These conditions occur often, most commonly when dealing with
ramified quadratic extensions of quotients of local principal
ideal domains with residue field of characteristic not 2 (see
Example \ref{tresdos} for more details). A distinguishing feature
of this case, as opposed to that of \cite{CHQS}, is that $h(v,v)$
is a non-unit for every $v\in V$. In particular, $V$ lacks an
orthogonal basis. As the existence of an orthogonal basis is the
building block of the theory developed in \cite{CHQS}, virtually
all arguments from \cite{CHQS} become invalid under the present
circumstances. Moreover, it turns out that now $n=2m$ must be even
and, when $A$ is finite, $U_{2m}(A)$ is an extension of the
\emph{symplectic} group $\Sp_{2m}(q)$ defined over the residue
field $F_q=A/\r$. In view of the essential differences between the
present case and that of \cite{CHQS}, we hereby develop, from the
beginning, the tools required to compute $|U_{2m}(A)|$ when $A$ is
finite and the above hypotheses apply. In particular, a detailed
and simple proof that the reduction homomorphism is surjective is
given.

The paper is essentially self-contained and its contents are as
follows. It is shown in \S\ref{s1} that $n=2m$ must be even, and
that $V$ admits a basis relative to which the Gram matrix of $h$
is equal to
$$
J=\left(
    \begin{array}{cc}
      0 & 1 \\
      -1 & 0 \\
    \end{array}
  \right),
$$
where all blocks have size $m\times m$. Thus, $U_{2m}(A)$ consists of all $X\in \GL_{2m}(A)$ such that
$$
X^*JX=J,
$$
where $(X^*)_{ij}=(X_{ji})^*$. We prove in \S\ref{s2} that
$U_{2m}(A)$ acts transitively on basis vectors of $V$ having the
same length. An important tool is found in \S\ref{s3}, namely the
fact that the canonical reduction map $U_{2m}(A)\to U_{2m}(A/\i)$
is surjective, where $\i$ is a $*$-invariant ideal of $A$ (the
proof of the corresponding result from \cite{CHQS} makes extensive
use of the fact that an orthogonal basis exists). The surjectivity
of the reduction map allows us to compute, in \S\ref{s4}, the
order of $U_{2m}(A)$ when $A$ is finite, by means of a series of
reductions (a like method was used in \cite{CHQS}). We find that
\begin{equation}\label{or1}
|U_{2m}(A)|=|\r|^{2m^2-m}|\m|^{2m} |\Sp_{2m}(q)|,
\end{equation}
where $\m=R\cap\r$ and $R$ is the additive group of all $a\in A$
such that $a^*=a$. We also obtain in~\S\ref{s4} the order of the
kernel, say $U_{2m}(\i)$, of the reduction map $U_{2m}(A)\to
U_{2m}(A/\i)$. Here $U_{2m}(\i)$ consists of all $1+X\in
U_{2m}(A)$ such that $X\in M_{2m}(\i)$. When $\i\neq A$, we obtain
\begin{equation}\label{or3}
U_{2m}(\i)=|\i|^{2m^2-m}|\i\cap\m|^{2m}.
\end{equation}
A totally independent way of computing $|U_{2m}(A)|$ is offered in
\S\ref{s5}, where we show that
\begin{equation}\label{or2}
|U_{2m}(A)|=\frac{|\r|^{m(m+1)}|A|^{m^2} (q^{2m}-1)(q^{2(m-1)}-1)\cdots (q^2-1)}{|S|^{2m}}.
\end{equation}
Here $S$ stands for the additive group of all $a\in A$ such that
$a^*=-a$. It should be noted that (\ref{or1})-(\ref{or2}) are
valid even when $A$ is neither commutative nor principal. We also
prove in \S\ref{s5} that the number of basis vectors of $V$ of any
given length is independent of this length and equal~to
$$
(|A|^{2m}-|\r|^{2m})/|S|.
$$
In this regard, in \S\ref{s6} we demonstrate that the order of the
stabilizer, say $S_v$, in $U_{2m}(A)$ of a basis vector $v$ is
independent of $v$ and its length, obtaining
$$
|S_v|=|U_{2(m-1)}(A)|\times |A|^{2m-1}/|S|.
$$
We end the paper in \S\ref{s7}, where a refined version of  (\ref{or1}) and (\ref{or2}) is given when $A$ is commutative and principal.

Virtually all the above material will find application in a forthcoming paper on the Weil representation of $U_{2m}(A)$.

\section{Non-degenerate skew-hermitian forms}\label{s1}

Let $A$ be a ring with $1\neq 0$. The Jacobson radical of $A$ will be denoted by $\r$ and the unit group of $A$ by $U(A)$.
We assume that $A$ is endowed with an involution $*$, which we interpret to mean an antiautomorphism of order $\leq 2$. Note that if $*=1_A$
then $A$ is commutative. Observe also that $\r$ as well as all of its powers are $*$-invariant ideals of $A$.

We fix a right $A$-module $V$ and we view its dual $V^*$ as a right $A$-module via
$$
(\alpha a)(v)=a^* \alpha(v),\quad v\in V,a\in A,\alpha\in V^*.
$$
We also fix a skew-hermitian form on $V$, that is, a function $h:V\times V\to A$ that is linear in the second variable
and satisfies
$$
h(v,u)=-h(u,v)^*,\quad u,v\in V.
$$
In particular, we have
$$
h(u+v,w)=h(u,w)+h(v,w)\text{ and }h(ua,v)=a^*h(u,v),\quad u,v,w\in V,a\in A.
$$
Associated to $h$ we have an $A$-linear map $T_h:V\to V^*$, given by
$$
T_h(v)=h(v,-),\quad v\in V.
$$
We assume that $h$ is non-degenerate, in the sense that $T_h$ is an isomorphism.

\begin{lemma}\label{deco} Suppose $V=U\perp W$, where $U,W$ are submodules of $V$. Then the restriction $h_U$ of $h$ to $U\times U$ is non-degenerate.
\end{lemma}

\begin{proof} Suppose $T_{h_U}(u)=0$ for some $u\in U$. Since $h(u,W)=0$ and $V=U+W$, it follows that $h(u,V)=0$, so $u=0$ by the non-degeneracy of $h$. This proves that $T_{h_U}$ is injective.

    Suppose next that $\al\in U^*$. Extend $\al$ to $\be\in V^*$ via $\be(u+w)=\al(u)$. Since $h$ is non-degenerate, there is $v\in V$ such that
$\be=T_h(v)$. Now $v=u+w$ for some $u\in U$ and $w\in W$. We claim that $\al=T_{h_U}(u)$. Indeed, given any $z\in U$, we have
$$
[T_{h_U}(u)](z)=h_U(u,z)=h(u,z)=h(u+w,z)=h(v,z)=\be(z)=\al(z).
$$
\end{proof}

The Gram matrix $M\in M_k(A)$ of a list of vectors $v_1,\dots,v_k\in V$
is defined by $M_{ij}=(v_i,v_j)$.

\begin{lemma}\label{gram} Suppose the Gram matrix of $u_1,\dots,u_k\in A$, say $M$, is invertible.
Then $u_1,\dots,u_k$ are linearly independent.
\end{lemma}

\begin{proof} Suppose $u_1a_1+\cdots+u_ka_k=0$. Then $h(u_i,u_1)a_1+\cdots+h(u_i,u_k)a_k=0$ for every $1\leq i\leq k$.
Since $M$ is invertible, we deduce that $a_1=\cdots=a_k=0$.
\end{proof}

We make the following assumptions on $A$ for the remainder of the paper:

\medskip

(A1) $A$ is a local ring (this means that $A/\r$ is a division ring or, alternatively, that every element of $A$ is either in $U(A)$ or in $\r$).

(A2) $2\in U(A)$.

(A3) If $a\in A$ and $a^*=-a$ then $a\in\r$; in particular, $b-b^*\in\r$ for all $b\in A$.

(A4) $\r$ is nilpotent; the nilpotency degree of $\r$ will be denoted by $e$.

\begin{exa}\label{tresdos}{\rm Let $B$ be a local commutative ring with nilpotent Jacobson radical $\b$. Suppose $2\in U(B)$ and let $\si$
be an automorphism of $B$ of order $\leq 2$. Consider
the twisted polynomial ring $C=B[t; \si]$. Then $C$ has a unique involution $*$ that sends $t$ to $-t$ and fixes every $b\in B$. Thus
$$
(b_0+b_1t+b_2 t^2+b_3 t^3+\cdots)^*=b_0-b_1^\si t+b_2 t^2-b_3^\si t^3+\cdots{\rm ( finite\; sum)}
$$
Given any $b\in\b$, the ideal $(t^2-b)$ of $C$ is $*$-invariant, so the quotient ring $A=C/(t^2-b)$ inherits an involution from $C$ and satisfies all our requirements. Note that if $\si\neq 1_B$ then $A$ is not commutative. Two noteworthy special cases are the following:

$\bullet$ Let $D$ be a local principal ideal domain with finite residue field of odd characteristic and let $B$ be a quotient of $D$ by a positive power of its maximal ideal; take $\si=1_B$ and let $b$ be a generator of $\b$. Then $A$ is a finite, commutative, principal ideal, local ring.

$\bullet$ Let $B$ be a finite, commutative, principal ideal, local ring with Jacobson radical $\b$. Suppose $2\in U(B)$ and let $\si\neq 1$
be an automorphism of $B$ of order 2 (as an example, take $A$ and its involution,
as in the previous case). Take $b$ to be a generator of $\b$ and let $a=t+(t^2-b)\in A$. Then $Aa=aA$ is the Jacobson radical of $A$. Moreover, every left
(resp. right) ideal of $A$ is
a power of $Aa$ (and hence an ideal).  Furthermore, note that $A/Aa\cong B/Bb$.}
\end{exa}

We will also make the following assumption on $V$ for the remainder of the paper:

\medskip

(A5) $V$ is a free $A$-module of finite rank $n>0$ (reducing $V$ modulo $\r$, we see the rank of $A$ is well-defined).

In what follows we  write $(u,v)$ instead of $h(u,v)$.

\begin{lemma}\label{basis} Any linearly independent list of vectors from $V$
is part of a basis; if the list has $n$ vectors, it is already a basis, and no list has more than $n$ vectors.
\end{lemma}

\begin{proof} Suppose $u_1,\dots,u_k\in V$ are linearly independent and
$v_1,\dots,v_n$ is a basis of $V$. By (A4) there is some $r\neq 0$
in $\r^{e-1}$ such that $\r r=0$. Write
$u_1=v_1a_1+\cdots+v_na_n$, where $a_i\in A$. If all $a_i\in\r$
then $u_1r=0$, contradicting linear independence. By (A1), we may assume
without loss of generality that $a_1\in U(A)$, whence
$u_1,v_2,\dots,v_n$ is a basis of $V$. Next write
$u_2=u_1b_1+v_2b_2+\cdots+v_nb_n$, where $b_i\in A$. If $b_i\in\r$
for all $i>1$ then $u_1(-b_1r)+u_2r=0$, contradicting linear
independence. This process can be continued and the result
follows.
\end{proof}

\begin{cor}\label{bas} If $v_1,\dots,v_n$ is a basis of $V$, $\i$ is a proper ideal of $A$, and $v_i\equiv u_i\mod V\i$ for all $1\leq i\leq n$, then $u_1,\dots,u_n$ is also a basis of $V$.
\end{cor}

\begin{proof} Let $M$ (resp. $N$) be the Gram matrix of $v_1,\dots,v_n$ (resp. $u_1,\dots,u_n$). Then, by assumption, $N=M+P$, for some
$P\in M_n(\i)$, so $P\in M_n(\r)$ by (A1). It is well-known (see
\cite[Theorem 1.2.6]{H}) that $M_n(\r)$ is the Jacobson radical of $M_n(A)$. On the other hand, by assumption and Lemma \ref{gram}, $M\in\GL_n(A)$. It follows that $N\in\GL_n(A)$ as well.
\end{proof}

\begin{lemma}\label{basin} Let $v_1,\dots,v_n$ be a basis of $V$. Then, given any $1\leq i\leq n$, there exists $1\leq j\leq n$
such that $j\neq i$ and $(v_i,v_j)\in U(A)$.
\end{lemma}

\begin{proof} Suppose, if possible, that $(v_i,v_j)\notin U(A)$ for all $j\neq i$. Then by (A1), $(v_i,v_j)\in \r$ for all $j\neq i$.
Since $(v_i,v_i)\in\r$ by (A3), it follows that
$(v_i,V)\in\r$. As $v_1,\dots,v_n$ span $V$ and $T_h:V\to V^*$ is surjective, every linear functional $V\to A$
has values in $\r$, a contradiction.
\end{proof}

\begin{cor}\label{basincor} If $v\in V$ is a basis vector, there is some $w\in V$ such that $(v,w)=1$.\qed
\end{cor}

\begin{lemma}\label{uno} Let $\i$ be a proper ideal of $A$. Suppose $v\in V$ is a basis vector such that $(v,v)\in \i$.
Then there is a basis vector $z\in V$ such that $v\equiv z\mod V\i$ and $(z,z)=0$.
\end{lemma}

\begin{proof} It follows from (A1) and (A4) that $\i$ is nilpotent, and we argue by induction on the nilpotency degree, say $f$, of $\i$. If $f=1$ then $\i=0$ and we take $z=v$. Suppose $f>1$
and the result is true for ideals of nilpotency degree less than $f$.
By Corollary \ref{basincor}, there is $w\in V$ such that $(v,w)=1$. Observing that the Gram matrix of $v,w$ is invertible,
Lemmas \ref{gram} and \ref{basis} imply that $v,w$ belong to a common basis of $V$.
Then, for any $b\in A$, $v+wb$ is also a basis vector and, moreover,
$$
(v+wb,v+wb)=(v,v)+b^*(w,w)b+b-b^*.
$$
Thanks to (A2), we may take $b=-(v,v)/2\in \i$. Then $b^*=-b$, so $b-b^*=-(v,v)$ and
$$
(v+wb,v+wb)=b^*(w,w)b=-b(w,w)b\in \i^2.
$$
As the nilpotency degree of $\i^2$ is less than $f$, by induction hypothesis
there is a basis vector $z\in V$ such that $v+wb\equiv z\mod V\i^2$ and $(z,z)=0$. Since $v\equiv v+wb\equiv z\mod V\i$, the result follows.
\end{proof}

\begin{lemma}\label{tres} Let $\i$ be a proper ideal of $A$. Suppose $u,v\in V$ satisfy
$(u,v)\equiv 1\mod \i$. Then there is $z\in V$ such that $z\equiv v\mod V\i$ and $(u,z)=1$.
\end{lemma}

\begin{proof}
    Since $(u,v) \equiv 1 \mod \i$, $(u,v)$ must be a unit. Moreover,
    $(u,v)^{-1} \equiv 1 \mod \i$, so we can take $z = v(u,v)^{-1}$.
\end{proof}

\begin{lemma}\label{dos} Let $\i$ be an ideal of $A$. Suppose $u,v\in V$ satisfy $(u,u)=0$, $(u,v)=1$ and $(v,v)\in \i$.
Then there is $z\in V$ such that $z\equiv v\mod V\i$, $(u,z)=1$ and $(z,z)=0$.
\end{lemma}

\begin{proof} Set $b=(v,v)/2\in \i$ and $z=ub+v$. Then $b^*=-b$, so that $b^*-b=-(v,v)$ and
$$
(z,z)=(ub+v,ub+v)=b^*-b+(v,v)=0,\quad (u,z)=(u,u)b+(u,v)=1.
$$
\end{proof}

A symplectic basis of $V$ is a basis $u_1,\dots,u_m,v_1,\dots,v_m$ such that $(u_i,u_j)=0=(v_i,v_j)$ and
$(u_i,v_j)=\delta_{ij}$. A pair of vectors $u,v$ of $V$ is symplectic if $(u,v)=1$ and $(u,u)=0=(v,v)$.

\begin{lemma}\label{exte}
Suppose the Gram matrix, say $M$, of $v_1,\dots,v_s$ is invertible. If $s<n$ then
there is a basis $v_1,\dots,v_s,w_1,\dots,w_t$ of $V$ such that $(v_i,w_j)=0$ for all $i$ and $j$.
\end{lemma}

\begin{proof} By Lemmas \ref{gram} and \ref{basis}, there is a basis $v_1,\dots,v_s,u_1,\dots,u_t$ of $V$. Given $1\leq i\leq t$, we wish to find $a_1,\dots,a_s$
so that $w_i=u_i-(v_1a_1+\cdots+v_s a_s)$ is orthogonal to all $v_j$. This means
$$
(v_j,v_1)a_1+\cdots+(v_j,v_s)a_s=(v_j,u_i),\quad 1\leq j\leq s.
$$
This linear system can be solved by means of $M^{-1}$. Since $v_1,\dots,v_s,w_1,\dots,w_t$ is a basis of $V$,
the result follows.
\end{proof}

\begin{prop}\label{sp} $V$ has a symplectic basis; in particular $n=2m$ is even.
\end{prop}

\begin{proof} Let $w_1,\dots,w_n$ be a basis of $V$, whose existence is guaranteed by (A5). We argue by induction on $n$.

By (A3), $(w_1,w_1)\in\r$. We infer from Lemma \ref{uno} the existence of a basis vector $u\in V$ such that $(u,u)=0$.
By Corollary \ref{basincor}, there is $w\in V$ such that $(u,w)=1$. By Lemma \ref{dos}, there
is $v\in V$ such that $(v,v)=0$ and $(u,v)=1$.

It follows from Lemmas \ref{gram} and \ref{basis} that $u,v$ is part of a basis of $V$.
If $n=2$ we are done. Suppose $n>2$ and the result is true for smaller ranks. By Lemma \ref{exte} there is basis $u,v,w_1,\dots,w_{n-2}$ of $V$
such that every $w_i$ is orthogonal to both $u$ and $v$. Let $U$ be the span of $u,v$ and let $W$ be the span of $w_1,\dots,w_{n-2}$.
By Lemma \ref{deco}, the restriction of $h$ to $W$ is non-degenerate. By inductive assumption, $n-2$ is even, say $n-2=2(m-1)$, and $W$ has a symplectic basis, say $u_2,\dots,u_m,v_2,\dots,v_m$. It follows that $u_1,u_2,\dots,u_m,v,v_2,\dots,v_m$ is a symplectic basis of $V$.
\end{proof}

\begin{cor} All non-degenerate skew-hermitian forms on $V$ are equivalent.\qed
\end{cor}

\section{The unitary group acts transitively on basis vectors of the same length}\label{s2}

By definition, the unitary group associated to $(V,h)$ is the subgroup, say $U(V,h)$, of $\GL(V)$ that preserves~$h$. Thus, $U(V,h)$  consists of all $A$-linear automorphisms
$g:V\to V$ such that $h(gu,gv)=h(u,v)$ for all $u,v\in V$.

\begin{theorem}\label{actra} Let $u,v\in V$ be basis vectors satisfying $(u,u)=(v,v)$. Then there exists $g\in U(V,h)$ such that $gu=v$.
\end{theorem}

\begin{proof} As $u$ is a basis vector, there is a vector $u'\in V$ such that $(u,u')=1$.
Let $W$ be the span of $u,u'$. By Lemma \ref{exte}, $V=W\oplus W^\perp$. The restrictions of $h$ to $W$ and $W^\perp$
are non-degenerate by Lemma \ref{deco}. A like decomposition exists for $v$. Thus, by means of Proposition \ref{sp},
we may restrict to the case $n=2$.

By Proposition \ref{sp}, there is a symplectic basis $x,y$ of $V$ and we have $u=xa+yb$ for some $a,b\in A$. Since $u$
is a basis vector, one of these coefficients, say $a$ is a unit. Replacing $x$ by $xa$ and $y$ by $y(a^*)^{-1}$,
we may assume that $u=x+yb$ for some $b\in A$, where $x,y$ is still a symplectic basis of $V$. We have
$b-b^*=(u,u)$. Likewise, there is a symplectic basis $w,z$ of $V$ such that
$v=w+zc$, where $c-c^*=(v,v)=(u,u)=b-b^*$. It follows that $c=b+r$ for some $r\in A$ such that $r^*=r$.

Replace $w,z$ by $w-zr,z$. This basis of $V$ is also symplectic, since $(w-zr,w-zr)=0$ (because $r-r^*=0$). Moreover, $v=(w-zr)+z(c+r)=(w-rz)+zb$. Thus, $u$ and $v$ have
exactly the same coordinates, namely $1,b$ relative to some symplectic bases of $V$. Let $g\in U$ map one symplectic basis into the other one.
Then $gu=v$, as required.
\end{proof}

\section{Reduction modulo a $*$-invariant ideal}\label{s3}

\begin{lemma}\label{util} Let $\i$ be a proper ideal of $A$. Suppose $w_1,\dots,w_m,z_1,\dots,z_m\in V$ satisfy
$$
(w_i,z_j)\equiv\delta_{ij}\mod\i,\; 1\leq i\leq m,\quad (w_i,w_j)\equiv 0\equiv (z_i,z_j)\mod\i,\; 1\leq i,j\leq m.
$$
Then there exists a symplectic basis $w'_1,\dots,w'_m,z'_1,\dots,z'_m$ of $V$ such that
$$
w_i\equiv w'_i\mod V\i,\quad z_i\equiv z'_i\mod V\i,\; 1\leq i\leq m.
$$
\end{lemma}

\begin{proof} By induction on $m$. Successively applying Lemmas \ref{uno}, \ref{tres} and \ref{dos} we obtain
$w'_1,z'_1\in V$ such that
$$
w_1\equiv w'_1 \mod V\i,\quad z_1\equiv z'_1 \mod V\i,\quad (w'_1,w'_1)=0=(z'_1,z'_1),\quad (w'_1,z'_1)=1.
$$
If $m=1$ we are done. Suppose $m>1$ and the result is true for smaller ranks. Applying Corollary \ref{bas}, we see that
$w'_1,z'_1,w_2\dots,w_m,z_2,\dots,z_m$ is a basis of $V$. Applying the procedure of Lemma \ref{exte} we obtain a basis
$w'_1,z'_1,w^0_2\dots,w^0_m,z^0_2,\dots,z^0_m$ of $V$ such that $w'_1,z'_1$ are orthogonal to all other vectors in this list.
Since $(x,y)\in\i$ when $x\in\{w'_1,z'_1\}$ and $y\in\{w_2\dots,w_m,z_2,\dots,z_m\}$, the proof of Lemma \ref{exte} shows that
$$
w^0_i\equiv w_i\mod V\i,\quad z^0_i\equiv z_i\mod V\i,\quad 2\leq i\leq m.
$$
Therefore
$$
(w^0_i,z^0_j)\equiv\delta_{ij}\mod\i,\; 2\leq i\leq m,\quad (w^0_i,w^0_j)\equiv 0\equiv (z^0_i,z^0_j)\mod\i,\; 2\leq i,j\leq m.
$$
By Lemma \ref{deco}, the restriction of $h$ to the span, say $W$, of $w^0_2\dots,w^0_m,z^0_2,\dots,z^0_m$ is non-degenerate.
By induction hypothesis, there is a symplectic basis $w'_2,\dots,w'_m,z'_2,\dots,z'_m$ of $W$ such that
$$
w'_i\equiv w^0_i\mod V\i, \quad z'_i\equiv z^0_i\mod V\i,\; 2\leq i\leq m.
$$
Then $w'_1,\dots,w'_m,z'_1,\dots,z'_m$ is a basis of $V$ satisfying all our requirements.
\end{proof}

Let $\i$ be a $*$-invariant ideal of $A$. Then $\overline{A}=A/\i$ inherits an involution, also denoted by $*$, from $A$ by declaring
$(a+\i)^*=a^*+\i$. This is well-defined, since $\i$ is $*$-invariant. Set $\overline{V}=V/V\i$ and consider the skew-hermitian
form $\overline{h}:\overline{V}\times \overline{V}\to \overline{A}$, given by $\overline{h}(u+V\i,v+V\i)=h(u,v)+\i$. We see that
$\overline{h}$ is well-defined and non-degenerate. We then have a group homomorphism $U(V,h)\to U(\overline{V},\overline{h})$, given by $g\mapsto \overline{g}$, where $\overline{g}(u+V\i)=g(u)+V\i$.

\begin{theorem}\label{sur} Let $\i$ be a proper $*$-invariant ideal of $A$. Then the canonical group homomorphism $U(V,h)\to U(\overline{V},\overline{h})$ is surjective.
\end{theorem}

\begin{proof} Let $f\in U(\overline{V},\overline{h})$. By Proposition \ref{sp}, $V$ has a symplectic basis $u_1,\dots,u_m,v_1,\dots,v_m$.
We have $f(u_i+V\i)=w_i+V\i$ and $f(v_i+V\i)=z_i+V\i$ for some $w_i,z_i\in V$ and $1\leq i\leq m$.
Since $f$ preserves $\overline{h}$, we must have
$$
(w_i,z_j)\equiv\delta_{ij}\mod\i,\; 1\leq i\leq m,\quad (w_i,w_j)\equiv 0\equiv (z_i,z_j)\mod\i,\; 1\leq i,j\leq m.
$$
By Lemma \ref{util} there is a symplectic basis  $w'_1,\dots,w'_m,z'_1,\dots,z'_m$ of $V$ such that
$$
w_i\equiv w'_i\mod V\i,\quad  z_i\equiv z'_i\mod V\i,\; 1\leq i\leq m.
$$
Let $g\in U(V,h)$ map $u_1,\dots,u_m,v_1,\dots,v_m$ into $w'_1,\dots,w'_m,z'_1,\dots,z'_m$. Then $\overline{g}=f$.
\end{proof}

\section{Computing the order of the $U_{2m}(A)$ by successive reductions}\label{s4}

We refer to an element $a$ of $A$ as hermitian (resp. skew-hermitian) if $a=a^*$ (resp. $a=-a^*$).
Let $R$ (resp. $S$) be the subgroup of the additive group of $A$
of all hermitian (resp. skew-hermitian) elements. We know by (A3) that
\begin{equation}
\label{sr}
S\subseteq \r.
\end{equation}
Moreover, it follows from (A2) that
\begin{equation}
\label{ars}
A=R\oplus S.
\end{equation}
Letting $$\m=R\cap\r,$$
we have a group imbedding $R/\m\hookrightarrow A/\r$. In fact, we deduce from (\ref{sr}) and (\ref{ars}) that
\begin{equation}
\label{ars2}
A/\r\cong R/\m.
\end{equation}

\begin{lem}
\label{cuad}
 Suppose $\i$ is a $*$-invariant ideal of $A$ satisfying $\i^2=0$. Let $\{v_1,\dots,v_n\}$ be a basis of $V$
and let $J$ be the Gram matrix of $v_1,\dots,v_n$. Then, relative to $\{v_1,\dots,v_n\}$, the
kernel of the canonical epimorphism $U(V,h)\to U(\overline{V},\overline{h})$ consists
of all matrices $1+M$, such that $M\in M_n(\i)$ and
\begin{equation}
\label{anr}
M^*J+JM=0.
\end{equation}
\end{lem}

\begin{proof} By definition the kernel of $U(V,h)\to U(\overline{V},\overline{h})$ consists of all matrices of the form $1+M$,
where $M\in M_n(\i)$ and
$$
(1+M)^* J(1+M)=J.
$$
Expanding this equation and using $\i^2=0$ yields (\ref{anr}).
\end{proof}

Let $\{u_1,\dots,u_m,v_1,\dots,v_m\}$ be a symplectic basis of $V$. We write $U_{2m}(A)$ for the image of $U(V,h)$ under the group isomorphism $GL(V)\to\GL_{2m}(A)$ relative to $\{u_1,\dots,u_m,v_1,\dots,v_m\}$.

We make the following assumption on $A$ for the remainder of the paper:

\medskip

(A6) $A$ is a finite ring.

\medskip

We deduce from (\ref{ars}) that
$$|A|=|R||S|.$$
On the other hand, it follows from (A1), (A2) and (A6) that $F_q=A/\r$ is a finite field of odd characteristic. By (A3), $a+\r=a^* +\r$ for all $a\in A$, so the involution that $*$ induces
on $F_q$ is the identity. Taking $\i=\r$ in Theorem \ref{sur}, we have ${U}_{2m}(\overline{A})=\Sp_{2m}(q)$,
the symplectic group of rank $2m$ over $F_q$. Recall \cite[Chapter 8]{T}
that
$$
|\Sp_{2m}(q)|=(q^{2m}-1)q^{2m-1}(q^{2(m-1)}-1)q^{2m-3}\cdots (q^2-1)q=q^{m^2}(q^{2m}-1)(q^{2(m-1)}-1)\cdots (q^{2}-1).
$$

\begin{cor}
\label{cuad2} Suppose $\i$ is a $*$-invariant ideal of $A$ satisfying $\i^2=0$. Then the kernel of
$U(V,h)\to U(\overline{V},\overline{h})$ has order $|\i|^{2m^2-m}|\i\cap \m|^{2m}$.
\end{cor}

\begin{proof} Let $\{u_1,\dots,u_m,v_1,\dots,v_m\}$ be a symplectic basis of $V$. Thus, the Gram matrix, say $J\in M_{2m}(A)$, of
$u_1,\dots,u_m,v_1,\dots,v_m$ is
$$
J=\left(
    \begin{array}{cc}
      0 & 1 \\
      -1 & 0 \\
    \end{array}
  \right),
$$
where all blocks are in $M_{m}(A)$. According to Lemma \ref{cuad}, the kernel of $U(V,h)\to U(\overline{V},\overline{h})$ consists of all $1+M$, where
$$
M=\left(
    \begin{array}{cc}
      P & Q \\
      T & S \\
    \end{array}
  \right),
$$
and $S=-P^*$, $Q=Q^*$ and $T=T^*$, which yields the desired result.
\end{proof}

\begin{thm}
\label{zxz} Let $A$ be a finite local ring, not necessarily commutative, with Jacobson radical~$\r$ and residue field $F_q$ of odd characteristic. Suppose $A$ has
an involution $*$ such that $a-a^*\in\r$ for all $a\in A$. Let $\m$ be the group of all $a\in\r$ such that
$a=a^*$. Then
$$
|U_{2m}(A)|=|\r|^{2m^2-m}|\m|^{2m} q^{m^2}(q^{2m}-1)(q^{2(m-1)}-1)\cdots (q^2-1).
$$
\end{thm}

\begin{proof} Consider the rings
$$
A=A/\r^{e}, A/\r^{e-1},\dots,A/\r^2,A/\r.
$$
Each of them is a factor of $A$, so is local and inherits an
involution from $*$. Each successive pair is of the form
$C=A/\r^k,D=A/\r^{k-1}$,  where the kernel of the canonical
epimorphism $C\to D$ is $\j=\r^{k-1}/\r^k$, so that $\j^2=0$. We
may thus apply Theorem \ref{sur} and Corollary \ref{cuad2} $e-1$
times to obtain the desired result, as follows. We have
$$
|\r|=|\r^{e-1}/\r^e|\cdots |\r/\r^2|
$$
and
$$
|\m|=|\m\cap\r^{e-1}/\m\cap \r^e|\cdots |\m\cap\r^{k-1}/\m\cap \r^{k}|\cdots |\m\cap \r/\m\cap \r^2|,
$$
where the group of hermitian elements in the kernel of $C\to D$
has $|\m\cap\r^{k-1}/\m\cap \r^{k}|$ elements. Indeed, these
elements are those $a+\r^k$ such that $a\in \r^{k-1}$ and
$a-a^*\in \r^k$. But $a+a^*$ is hermitian, so
$a+a^*\in\m\cap\r^{k-1}$. Thus $$a=(a-a^*)/2+(a+a^*)/2\in
\r^k+\m\cap\r^{k-1}.$$ Hence the group of hermitian elements in
the kernel of $C\to D$ is
$$
(\m\cap\r^{k-1}+\r^k)/\r^k\cong \m\cap\r^{k-1}/(\m\cap\r^{k-1}\cap \r^k)\cong \m\cap\r^{k-1}/\m\cap\r^{k}.
$$
\end{proof}

\begin{thm}  Given a $*$-invariant proper ideal $\i$ of $A$, the kernel of
$U(V,h)\to U(\overline{V},\overline{h})$ has order $|\i|^{2m^2-m}|\i\cap\m|^{2m}$.
\end{thm}

\begin{proof} By Theorems \ref{sur} and \ref{zxz}, the alluded kernel has order
$$
\frac{|U(V,h)|}{|U(\overline{V},\overline{h})|}=\frac{|\r|^{2m^2-m}|\m|^{2m}}{|\r/\i|^{2m^2-m}|(\m+\i)/\i|^{2m}}=|\i|^{2m^2-m}|\i\cap\m|^{2m}.
$$
\end{proof}
\section{Computing the order of $U_{2m}(A)$ by counting symplectic pairs}\label{s5}

The following easy observation will prove useful. Given $s\in S$ and $y\in A$, we have
\begin{equation}
\label{yx}
y-y^*=s\text{ if and only if }y\in s/2+R.
\end{equation}

By the length of a vector $v\in V$ we understand the element $(v,v)\in S$. Given $s\in S$, the number of basis vectors
of $V$ of length $s$ will be denoted by $N(m,s)$.

\begin{lemma}\label{numbers} Given $s\in S$, we have
$$
N(1,s)=(|A|-|\r|)(|R|+|\m|)=(|A|^2-|\r|^2)/|S|,\quad s\in S.
$$
In particular, $N(1,s)$ is independent of $s$.
\end{lemma}

\begin{proof} Let $u,v$ be a symplectic basis of $V$. Given $(a,b)\in A^2$, the length of $w=ua+vb$ is
$$
(w,w)=a^* b-b^* a.
$$
Thus, we need to count the number of pairs $(a,b)\in A^2\setminus \r^2$ such that
\begin{equation}
\label{lens}
a^* b-b^* a=s.
\end{equation}
For this purpose, suppose first $a\in U(A)$. Setting $y=a^* b$ and using (\ref{yx}), we see that (\ref{lens}) holds
if and only if $b\in (a^*)^{-1}(s/2+R)$. Thus, the number of solutions $(a,b)\in A^2$ to (\ref{lens}) such that $a\in U(A)$ is
$(|A|-|\r|)|R|$. Suppose next that $a\notin U(A).$ Then $b \in U(A)$.
Rewriting (\ref{lens}) in the form
\begin{equation}
\label{lens2}
b^* a-a^* b=-s
\end{equation}
and setting $y=b^* a$, we see as above that (\ref{lens2}) holds if and only if $a\in (b^*)^{-1}(-s/2+R)$. Recalling that $a\in\r$, we
are thus led to calculating
$$
|[(b^*)^{-1}(-s/2+R)]\cap\r|=|(-s/2+R)\cap b^*\r|=|(-s/2+R)\cap \r|=|R\cap\r|,
$$
the last two equalities holding because $b\in U(R)$ and $s\in\r$. Recalling that $\m=R\cap\r$, it follows that
$N(1,s)=(|A|-|\r|)(|R|+|\m|)$. Since this is independent of $s$, we infer $N(1,s)=(|A|^2-|\r|^2)/|S|$.
\end{proof}

Note that the identity $(|A|-|\r|)(|R|+|\m|)=(|A|^2-|\r|^2)/|S|$ also follows from $|R|=|A|/|S|$ and $|\m|=|\r|/|S|$.

\begin{prop}\label{igual} Given $s\in S$, we have
$$
N(m,s)=(|A|^{2m}-|\r|^{2m})/|S|,\quad s\in S.
$$
In particular, $N(m,s)$ is independent of $s$.
\end{prop}

\begin{proof} The first assertion follows from the second. We prove the latter by induction on $m$. The case $m=1$ is done in Lemma
    \ref{numbers}. Suppose that $m>1$ and $N(m-1,s)$ is independent of $s$. Set $N=N(1,0)$ and $M=N(m-1,0)$. Decompose $V$ as $U\perp W$ where $U$ has rank 2. Thus $N(m,s)$ is the number of pairs $(u,w) \in U \times W$
such that either $w$ is an arbitrary element of $W$ and $u$ is a basis vector of $U$ of length $s-(w,w)$, or,
$u$ is an arbitrary non basis vector of $U$ and $w$ is a basis vector of $W$ of length $s - (u,u)$. These two
possibilities are mutually exclusive. It follows, using the inductive hypothesis, that
$$
N(m,s)=N|A|^{2(m-1)}+|\r|^2 M,
$$
which is independent of $s$.

\end{proof}

\begin{cor}\label{nusy} The number of symplectic pairs in $V$ is
$$
\frac{(|A|^{2m}-|\r|^{2m})|A|^{2m-1}}{|S|^2}.
$$
\end{cor}

\begin{proof} By Proposition \ref{igual}, the number of basis vectors of length 0 is
$(|A|^{2m}-|\r|^{2m})/|S|$. Given any such vector, say $u$, Lemma \ref{dos} ensures the existence of a vector $v\in V$ of length 0
such that $(u,v)=1$. Then, a vector $w\in V$ satisfies $(u,w)=1$ if and only if
$w=au+v+z$, where $z$ is orthogonal to $u,v$. Moreover, given any such $z$, we see that $w$ has length 0 if and only if
$$
0=(w,w)=(ua+v,ua+v)+(z,z)=a^*-a+(z,z).
$$
It follows from (\ref{yx}) that the number of solutions $a \in A$ to this equation is $|R|$.
We infer that the number of symplectic pairs is
$$
\frac{(|A|^{2m}-|\r|^{2m})}{|S|}\times |A|^{2m-2}|R|=\frac{(|A|^{2m}-|\r|^{2m})|A|^{2m-1}}{|S|^2}.
$$
\end{proof}

It follows from Lemma \ref{exte} and Proposition \ref{sp} that $U_{2m}(A)$ acts transitively on symplectic pairs. Moreover,
we readily see that the stabilizer of a given symplectic pair is isomorphic to $U_{2(m-1)}(A)$. We infer from Corollary
\ref{nusy} that
$$
|U_{2m}(A)|=\frac{(|A|^{2m}-|\r|^{2m})|A|^{2m-1}}{|S|^2}\times \frac{(|A|^{2(m-1)}-|\r|^{2(m-1)})|A|^{2m-3}}{|S|^2}\times\cdots\times
\frac{(|A|^{2}-|\r|^{2})|A|}{|S|^2}.
$$
We have proven the following result.
\begin{thm}
\label{zxz2} Let $A$ be a finite local ring, not necessarily commutative, with Jacobson radical~$\r$ and residue field $F_q$ of odd characteristic. Suppose $A$ has
an involution $*$ such that $a-a^*\in\r$ for all $a\in A$. Let $S$ be the  group of all $a\in A$ such that
$a=-a^*$. Then
$$
|U_{2m}(A)|=\frac{|\r|^{m(m+1)}|A|^{m^2} (q^{2m}-1)(q^{2(m-1)}-1)\cdots (q^2-1)}{|S|^{2m}}.\qed
$$
\end{thm}

\begin{note}{\rm We readily verify, by means of (\ref{ars}) and (\ref{ars2}), the equivalence of the formulae given in Theorems \ref{zxz} and \ref{zxz2}.}
\end{note}

\section{The order of the stabilizer of a basis vector}\label{s6}

\begin{thm}\label{lasta} Let $v\in V$ be a basis vector and let $S_v$ be the stabilizer of $v$ in $U(V,h)$. Then
$$
|S_v|=|U_{2(m-1)}(A)|\times |A|^{2m-1}/|S|.
$$
In particular, the order of $S_v$ is independent of $v$ and its length.
\end{thm}

\begin{proof} By Theorem \ref{actra}, the number of basis vectors of length $(v,v)$ is equal to $|U_{2m}(A)|/|S_v|$.
It follows from Proposition \ref{igual} that
\begin{equation}\label{idw}
|U_{2m}(A)|/|S_v|=(|A|^{2m}-|\r|^{2m})/|S|.
\end{equation}
On the other hand, the above discussion shows that
\begin{equation}\label{idw2}
|U_{2m}(A)|=\frac{(|A|^{2m}-|\r|^{2m})|A|^{2m-1}}{|S|^2}\times |U_{2(m-1)}(A)|.
\end{equation}
Combining (\ref{idw}) and (\ref{idw2}) we obtain the desired result.
\end{proof}

\section{The case when $A$ is commutative and principal}\label{s7}

We make the following assumptions on $A$ until further notice:

\medskip

(A7) There is $a\in\r$ such that $Aa=aA=\r$.

\medskip

(A8) The elements of $R$ commute among themselves.

\medskip

Using (A7), we see that $|A|=q^e$, $|\r|=q^{e-1}$. Moreover, from $Aa=aA$, we get $a^*A=Aa^*$. Since $a=(a-a^*)/2+(a+a^*)/2$, not both $a-a^*$ and $a+a^*$ can be in $\r^2$. Thus,
$\r$ has a generator $x$ that is hermitian or skew-hermitian and satisfies $Ax=xA$. In any case, $x^2$ is hermitian. We claim
that
$$
A=R+Rx.
$$
Note first of all that, because of (A8), $R$ is a subring of $A$. Clearly, $R$ is a local ring with maximal ideal $\m=R\cap\r$
and residue field $R/\m\cong A/\r$. Secondly, from $A = R+S$ and $S\subseteq\r = Ax$, we deduce
\begin{equation}
\label{der}
A=R+Ax.
\end{equation}
Repeatedly using (\ref{der}) as well as (A8), we obtain
$$
\begin{aligned}
A &=R+(R+Ax)x=R+Rx+Ax^2=R+Rx+(R+Ax)x^2\\
&=R+Rx+Ax^3=R+Rx+(R+Ax)x^3=R+Rx+Ax^4=\dots=R+Rx.
\end{aligned}
$$
If $*=1_A$ then $A=R$ and $\r=\m$ has $q^{e-1}$ elements. We make the following assumptions on $A$ until further notice:

\medskip

(A9) $*\neq 1_A$.

\medskip

(A10) $R\cap Rx=(0)$.

\medskip

It follows from (A9) and $A=R+Rx$ that $x$ cannot be hermitian. Therefore $x$ is skew-hermitian. Note that $R$ is a principal ring
with maximal ideal $\m=Rx^2$, since
$$
\m=R\cap Ax=R\cap (R+Rx)x=R\cap (Rx+Rx^2)=Rx^2+(R\cap Rx)=Rx^2.
$$

\begin{lemma}\label{evod} The group epimorphism $f:R\to Rx$, given by $f(r)=rx$, is injective if $e$ is even, whereas
the kernel of $f$ is $Rx^{e-1}$ and has $q$ elements if $e$ is odd.
\end{lemma}

\begin{proof} Note that every non-zero element of $A$ is of the form $cx^i$ for some unit $c\in U(A)$ and a unique $0\leq i<e$.
It follows that the annihilator of $x$ in $A$ is equal to $Ax^{e-1}$. From $A=R+Rx$, we infer $Ax^{e-1}=Rx^{e-1}$. Thus,
the kernel of $f$ is $R\cap Rx^{e-1}$. If $e$ is even then $R\cap Rx^{e-1}\subseteq R\cap Rx=(0)$, while if $e$ is odd
$$R\cap Rx^{e-1}=Rx^{e-1}=Ax^{e-1}$$ is a 1-dimensional vector space over $F_q=A/\r$.
\end{proof}

\begin{cor} We have
$$
|A|=|R|^2\text{ if }e\text{ is even and }|A|=\frac{|R|^2}{q}\text{ if }e\text{ is odd.}
$$
Thus, either $e=2\ell$ is even and
$$
|\r|=q^{2\ell-1},\;|\m|=q^{\ell-1}
$$
or $e=2\ell -1$ is odd and
$$
|\r|=q^{2\ell-2},\; |\m|=q^{\ell-1}.
$$
\end{cor}

\begin{proof} This follows from $A=R\oplus Rx$, Lemma \ref{evod} and the group isomorphism $R/\m\cong F_q$.
\end{proof}

We now resume the general discussion and note that if $A$ is a commutative, principal ideal ring and $*\neq 1_A$ then conditions (A7)-(A10)
are automatically satisfied, for in this case we have $R\cap Rx\subseteq R\cap S=(0)$. It is clear that $A\cong R[t]/(t^2-x^2)$ if $e=2\ell$ is even, and $A\cong R[t]/(t^2-x^2,t^{2\ell-1})$ if $e=2\ell-1$ is odd. Using part of the above information together with Theorem \ref{zxz}, we obtain the following result.

\begin{thm}
\label{ords} Let $A$ be a finite, commutative, principal, local ring with Jacobson radical~$\r$ and residue field $A/\r\cong F_q$ of odd characteristic. Let $e$ be the nilpotency degree of $\r$. Suppose $A$ has
an involution $*$ such that $a-a^*\in\r$ for all $a\in A$.

(a) If $*=1_A$ then
$$
|U_{2m}(A)|=|\Sp_{2m}(A)|=q^{(e-1)(2m^2+m)+m^2}(q^{2m}-1)(q^{2(m-1)}-1)\cdots (q^2-1).
$$
(b) If $*\neq 1_A$ and $e=2\ell$ is even then
$$
|U_{2m}(A)|=q^{(2\ell-1)(2m^2-m)}q^{2(\ell-1)m} q^{m^2}(q^{2m}-1)(q^{2(m-1)}-1)\cdots (q^2-1).
$$
(b) If $*\neq 1_A$ and $e=2\ell-1$ is odd then
$$
|U_{2m}(A)|=q^{(2\ell-2)(2m^2-m)}q^{2(\ell-1)m} q^{m^2}(q^{2m}-1)(q^{2(m-1)}-1)\cdots (q^2-1).\qed
$$
\end{thm}

\medskip

\begin{note}{\rm Our initial conditions on $A$ do not force $R$ to be a subring of $A$ or $R\cap Rx=(0)$. Indeed, let $A$ be as indicated in the parenthetical remark of the second case of Example \ref{tresdos}, and set $x=a$, $r=b$. Then $rx\in R\cap Rx$, so $R\cap Rx\neq (0)$, and $rxr=-r^2x$
with $(-r^2 x)^*=r^2 x$, so $R$ is not a subring of $A$. It is also clear that $A$ need not be principal, even if so is $R$, as can be seen
by taking $b=0$ and $B$ not a field in the general construction of Example \ref{tresdos} (e.g. $A=Z_{p^2}[t]/(t^2)$).}
\end{note}


\noindent{\bf Acknowledgement.} We are very grateful to the referee for a thorough reading of the paper and valuable suggestions.


\end{document}